\documentclass[12pt]{amsart}
\usepackage{amscd}
\usepackage{verbatim}
\usepackage{amssymb, amsmath, amsthm, amscd,ifthen}
\usepackage[cp866]{inputenc}

\usepackage{tikz}

\usepackage{amsmath,amssymb,amscd,}
\usepackage[all]{xy}

\usepackage{amssymb}

\newtheorem{thm}{Theorem}[section]

\newtheorem{prop}[thm]{Proposition}

\theoremstyle{definition}

\newtheorem{dfn}[thm]{Definition}

\title
{Around a conjecture by R. Connelly, E. Demaine, and G. Rote}

\author
       {A. Igamberdiev, G. Panina}


\begin{document}

\begin{abstract}
Denote by $M(P)$ the configuration space of a planar polygonal
linkage, that is, the space of all possible planar configurations
modulo congruences,  including configurations with
self-intersections. A particular interest  attracts its subset
$M^o(P) \subset M(P)$  of all configurations \emph{without}
self-intersections.  R. Connelly, E. Demaine, and G. Rote proved
that $M^o(P)$ is contractible and conjectured that so is its closure
$\overline{M^o(P)}$. We disprove this conjecture  by showing that a
special choice of $P$ makes the homologies $H_k(\overline{M^o(P)})$
non-trivial.

\end{abstract}
\maketitle \setcounter{section}{0}

\section{Introduction}
An \textit{ $n$-linkage} is a sequence of positive numbers
$l_1,\dots ,l_n$. It should be interpreted as a collection of rigid
bars of lengths $l_i$ joined consecutively by revolving joints in a
closed chain.

\begin{dfn}
 For a linkage $P$, \textit{a configuration} in the
Euclidean space $ \mathbb{R}^d$ is a sequence of points
$R=(a_1,\dots,a_{n}), \ a_i \in \mathbb{R}^d$ with
$l_i=|a_i,a_{i+1}|, \ n+1=1$.

 The set $M(P)$ of all such
configurations modulo the action of all isometries of $R^2$ is
\textit{the configuration space of the linkage} $P$.

It comes together with its subset $M^o(P) \subset M(P)$  of all
configurations \emph{without} self-intersections.

\end{dfn}

In \cite{CDR}  R. Connelly, E. Demaine, and G. Rote proved a
strengthened version of the famous carpenter's rule conjecture.
Namely, using \textit{expansive motions} they showed that $M^o(P) $
is contractible. In the same paper they conjectured that the closure
of $M^o(P) $ is also contractible.

We disprove the conjecture by showing that not only
$\overline{M^o(P)}$ can be non-contractible, but can also have other
non-trivial homologies.  For this, we use a simple trick which
produces non-contractible loops in $\overline{M^o(P)}$. To
understand the trick, it suffices to look at Fig. \ref{loop}.

However,  authors of \cite{CDR} were motivated in their study by a
physical model of a linkage which allows self-touching and
self-overlapping, but does not allow the edges pass one through
another, as it happens in our examples. It remains an open question
whether the space becomes contractible if we forbid such passes.

In this respect we mention two papers \cite{ADG, CDR1} where authors
treat the space of \textit{self-touching configurations}, that is,
configurations without transversal crossings. The authors equip the
space by some additional structure  which yields an ordering on
overlapping edges. In such a space the contractible loop introduced
in Section 2 becomes disconnected. In \cite{ADG}, it is proven that
the space of self-touching configuration $\mathcal{A}(P)$ (equipped
with additional structure) is connected. To the best of our
knowledge, nothing is known about contractibility of the space
$\mathcal{A}(P)$.

However, if we forget the additional structure, the set
$\mathcal{A}(P)$ does not coincide with the set $\overline{M^o(P)}$:
 a self-touching configuration does not necessarily
belong to $\overline{M^o(P)}$, see Fig. \ref{isolated}.

\begin{figure}
 \centering
\includegraphics[width=8 cm]{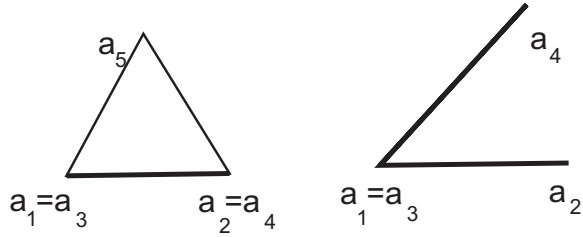}
\caption{These self-touching configurations do not belong to
$\overline{M^o(P)}$ }\label{isolated}
\end{figure}

\bigskip

\textbf{Acknowledgments.} The first author was supported by the
Chebyshev Laboratory (Department of Mathematics and Mechanics,
St.-Petersburg State University) under RF government grant
11.G34.31.0026.

\section{A non-contractible loop }
For a reader not acquainted with the homology theory, we start with
an elementary example.

\begin{prop}\label{prop} Assume that for a linkage $P$, we have

$l_1>l_2<l_3$, and
 $l_1-l_2+l_3<\sum_{j\ne 1,2,3}l_j$.

Then the space $\overline{M^o(P)}$ contains a non-contractible loop.
\end{prop}
\begin{proof}

\begin{figure}
 \centering
\includegraphics[width=8 cm]{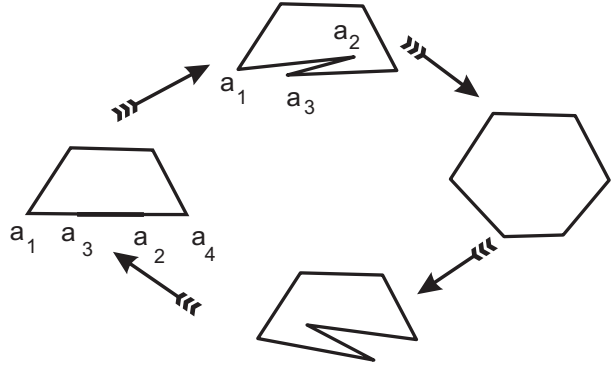}
\caption{A triple fold yields a non-contractible loop}\label{loop}
\end{figure}

Consider a continuous mapping $\alpha: M(P)\to S^1$ which maps a
configuration of the linkage to the value of the oriented angle
$\angle (a_1\ a_2\ a_3)$.

Next, consider  a  loop $\gamma(t):S^1\to \overline{M^o(P)}$ which
is depicted in Fig. \ref{loop}. All the points of $\gamma (S^1)$
except for exactly one (which corresponds to the configuration with
overlapping edges) lie in ${M^o(P)}$. An appropriate choice of
parameterization of $\gamma$ makes the following diagram commute:

$$
\xymatrix{
                                           & \overline{M^o(P)} \ar[dr]^{\alpha} &\\
S^1 \ar[ur]^{\gamma} \ar@{->}[rr]^{id}  & &S^1. }
$$
(Here and in the sequel, $id$ denotes the identity mapping.) This
means that the loop $\gamma$ is non-contractible.
\end{proof}

 In this respect, we conject  that for any linkage $P$, the space $$\overline{M^o(P)}
\setminus \{\hbox{configurations with triple folds}\}$$ is
contractible.

\section{Non-trivial homologies of the space $\overline{M^o(P)}$}
\begin{thm}
For every $m\in \mathbb N$, there exists a linkage $P$ such that for
all $k \leq n$, all  the homology groups $H_k(\overline{M^o(P)})$
are non-trivial.
\end{thm}
\begin{proof}
We construct a polygonal linkage with $n=4m$ edges, combining $m$
triple folds  from the previous section, see Fig. \ref{loops}.

Following the pattern of Section 2, we consider a continuous mapping
$$\alpha:\overline{M^o(P)}\to T^m=S^1\times ...\times S^1,$$
which maps a linkage $P$ to the string of oriented angles
$$\alpha(P)=(\angle(a_{1}\ a_2\ a_{3}),\angle(a_{6}\ a_7\ a_{8}),\cdots ).$$

Besides, analogously to the above, we get  a mapping

$$\gamma: T^m=S^1\times ...\times S^1\to \overline{M^o(P)}.$$

We may freely assume that the parameterization of  $\gamma$  makes
the following diagram commute:

$$
\xymatrix{
                                           & \overline{M^o(P)} \ar[dr]^{\alpha} &\\
T^m \ar[ur]^{\gamma} \ar@{->}[rr]^{id}  & &T^m .}
$$

 This immediately implies the following
commutative diagram for homology groups (see \cite{D}):

$$
\xymatrix{
                                           & H_k(\overline{M^o(P)}) \ar[dr]^{\alpha_k} &\\
H_k(T^m) \ar[ur]^{\gamma_k} \ar@{->}[rr]^{id}  & &H_k(T^m). }
$$

\bigskip

So, the group $H_k(T^m)$ (which is a free abelian group of rank
${\binom{m}{k}}$) is a subgroup of $H_k(\overline{M^o(P)})$.
\end{proof}

\begin{figure}
 \centering
\includegraphics[width=6cm]{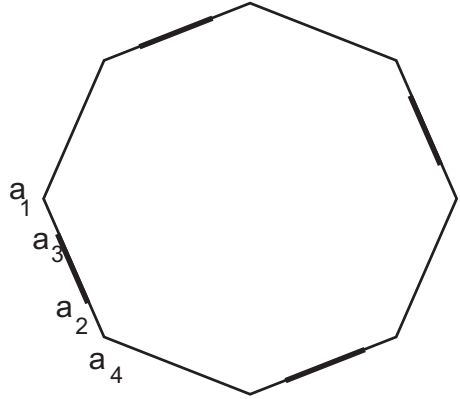}
\caption{$m$ triple folds yield $m$ non-homological
loops}\label{loops}
\end{figure}
\end{document}